\documentclass[12pt]{amsart}

\newtheorem{theorem}{Theorem}[section]

\newtheorem{lemma}[theorem]{Lemma}

\begin{document}

\title{Using Fibonacci factors to create Fibonacci pseudoprimes}

\author{Junhyun Lim}
\email{ limjunhyun@gmail.com}
\address{Department of Mathematics and
Computer Science, Santa Clara University, Santa Clara, CA 95053,
USA.}

\author{Shaunak Mashalkar}
\email{ssmash724@gmail.com}
\address{Department of Mathematics and
Computer Science, Santa Clara University, Santa Clara, CA 95053,
USA.}

\author{Edward F.~Schaefer}
\email{eschaefer@scu.edu}
\address{Department of Mathematics and
Computer Science, Santa Clara University, Santa Clara, CA 95053,
USA.}

\subjclass[2010]{Primary 11B39; Secondary 11A51}

\begin{abstract}
Carmichael showed  for sufficiently large $L$,
that $F_L$ has at least one prime divisor that is $\pm 1({\rm mod}\, L)$.
For a given $F_L$, we will show that a product of distinct
odd prime divisors with
that congruence condition is a Fibonacci pseudoprime. Such pseudoprimes can be used in an attempt, here unsuccessful, to find an example
of a Baillie-PSW pseudoprime, i.e.\ an odd Fibonacci pseudoprime that is congruent
to $\pm 2({\rm mod}\, 5)$ and is also a base-2 pseudoprime.
\end{abstract}

\thanks{The authors are grateful to Carl Pomerance for challenging us to the \$620 problem and useful conversations.}

\maketitle

\section{Introduction}

For all odd prime numbers $p$ we have $p|F_{p-(\frac{5}{p})}$, where
$(\frac{5}{p})$ is the Legendre symbol. This is well-known
and can be proved using the lemmas in Section~\ref{Proof}.
An odd composite integer $n$ is said to be a Fibonacci pseudoprime if $n|F_{n-(\frac{5}{n})}$ where $(\frac{5}{n})$ is the Jacobi symbol, which generalizes
the Legendre symbol.
Let us list the six smallest Fibonacci pseudoprimes, their prime factorizations,
and the smallest positive Fibonacci number that each divides. We have
$323=17\cdot 19|F_{18}$,
$377=13\cdot 29|F_{14}$,
$1891=31\cdot 61|F_{30}$,
$3827=43\cdot 89|F_{44}$,
$4181=37\cdot 113|F_{19}$, and
$5777=53\cdot 109|F_{27}$.
For each of the six smallest Fibonacci pseudoprimes $n$, if $F_L$ is the smallest
positive Fibonacci
number for which $n$ is a divisor (i.e.\ $L={\rm ord}_f(n)$), then the prime divisors of $n$ are each $\pm 1({\rm mod}\, L)$. This is not always the case. In fact the seventh smallest Fibonacci pseudoprime
is $6601 = 7\cdot 23\cdot 41|F_{120}$ and none of its prime divisors are $\pm 1({\rm mod}\, 120)$.
Nevertheless, we can use the above observation to create many Fibonacci pseudoprimes.
We construct them using the theorem below.

\begin{theorem}
\label{MainTheorem}
Let $L$ be a positive integer. Let $p_1, \ldots , p_k$, for some $k\geq 2$, be
distinct  odd primes dividing
$F_L$ with the property that for each $i$ we have $p_i\equiv \pm 1({\rm mod}\, L)$, assuming
at least two such primes exist.
Then $P:=\prod_{i=1}^k p_i$ is a Fibonacci pseudoprime.
\end{theorem}

It does seem quite common for prime divisors of $F_L$ to be  $\pm 1({\rm mod}\, L)$. For example, four of the prime divisors of $F_{100}$ are 101, 401, 3001, and 570601 and
$F_{7560}$ has at least 29 prime divisors that are $\pm 1({\rm mod}\, 7560)$
(note $F_{7560}$ has a composite factor with 711 decimal digits that is still unfactored).
All Fibonacci numbers up to $F_{1408}$
have been factored  completely
and complete or partial factorizations of $F_L$ for
$1409\leq L\leq 9999$ have been given
(see \cite{FF}).
For $1\leq L \leq 1408$, the average number of odd prime divisors of $F_L$ that are $\pm 1({\rm mod}\, L)$ is $\frac{7279}{1408}\approx 5.17$.

Applying Theorem~\ref{MainTheorem} to the fully and partially factored Fibonacci numbers $F_L$ for $1\leq L\leq 9999$, we can create
approximately $2^{31}$  Fibonacci pseudoprimes. However they will not all be distinct.
For example, $F_{19}=4181=37\cdot 113$. Not only are both prime divisors $\pm 1({\rm mod}\, 19)$,
they are both  $\pm 1({\rm mod}\, 38)$. So the Fibonacci pseudoprime 4181 will appear
for both $L=19$ and $L=38$.

One reason for the appearance of such prime divisors of Fibonacci numbers comes
from  \cite[Thm.\ XXVI]{Ca}. Carmichael proved that for every  $L\neq 1, 2, 5, 6, 12$ there is a prime divisor $p\equiv \pm 1({\rm mod}\, L)$ of
$F_L$ which divides no $F_K$ for $K<L$.
We leave it to the analytic number theorists to study the expected number of prime divisors
of $F_L$ that are $\pm 1({\rm mod}\, L)$ as a function of $L$.

In Section~\ref{Proof} we  prove Theorem~\ref{MainTheorem}.
It provides a new way of creating Fibonacci pseudoprimes and thus can be
added to the list of methods for constructing them found in  \cite{CG,DFM,Le,Pa,Ro,SW}.
There is much interest in finding an integer that is an odd Fibonacci
pseudoprime, congruent to $\pm 2({\rm mod}\, 5)$, and simultaneously a base-2 pseudoprime. These are sometimes referred to as Baillie-PSW pseudoprimes.
There is a heuristic argument that an example exists (see \cite{CG}, which refers
to ideas in \cite{Po1}).  However, no example
is known and there is a \$620 prize (payable by Carl Pomerance, Sam Wagstaff, and the
Number Theory Foundation, see \cite{Po2}) for those who either find an example, or prove
that none exists.
In Section~\ref{620} we
discuss how we used our theorem in our unsuccessful attempt at finding an example.

\section{Proof of the  Theorem}
\label{Proof}

We will use the following well-known lemmas.
Lemmas~\ref{pdivideFibo1} and \ref{pdivideFibo2} are proven  in \cite[p.\ 297]{Lu}

\begin{lemma}
\label{pdivideFibo1}
Let $p\equiv \pm 1({\rm mod}\, 5)$ be prime. Then $p|F_{p-1}$.\end{lemma}

\begin{lemma}
\label{pdivideFibo2}
Let $p\equiv \pm 2({\rm mod}\, 5)$ be prime. Then $p|F_{p+1}$.\end{lemma}

Lemma~\ref{Fibodivide} is proven in \cite[p.\ 35]{Ha}.

\begin{lemma}
\label{Fibodivide}
Let $m,n$ be positive integers. If $m|n$ then $F_m|F_n$.
\end{lemma}

Lemma~\ref{WhichFibDivByp} is proven in \cite[Thm.\ 3]{Wa}.

\begin{lemma}
\label{WhichFibDivByp}
 Let $n$ be a positive integer and let  $L$ be an integer such that $n|F_L$. Then ${\rm ord}_f(n)|L$.
\end{lemma}

Lemma~\ref{Legendre} follows from Gauss' law of quadratic reciprocity.

\begin{lemma}
\label{Legendre}
Let $p$ be an odd prime. If $p\equiv \pm 1({\rm mod}\, 5)$ then $(\frac{5}{p})=1$.
If $p\equiv \pm 2({\rm mod}\, 5)$ then $(\frac{5}{p})=-1$.
\end{lemma}

The following lemma is new.

\begin{lemma}
\label{PrimeDivisorsPM1ModL}
Let $p$ be an odd prime and $L$ be a positive integer with $p|F_L$.
Assume $p\equiv \pm 1({\rm mod}\, L)$.
Then $p\equiv (\frac{5}{p})({\rm mod}\, L)$.
\end{lemma}

\begin{proof}
Let $p=5$. Then from Lemmas~\ref{Fibodivide} and~\ref{WhichFibDivByp}, $p|F_L$ if and only if $5|L$.
So the result is vacuously true as $p\not\equiv \pm 1({\rm mod}\, L)$.

Assume $p\neq 5$ is an odd prime, that $p|F_L$, and $p\equiv \pm 1({\rm mod}\, L)$.
Let $p\equiv \pm 1({\rm mod}\, 5)$. From Lemma~\ref{pdivideFibo1} we have $p|F_{p-1}$.
Assume $p\equiv -1({\rm mod}\, L)$. Then $L|p+1$.
From Lemma~\ref{Fibodivide} we have $F_L|F_{p+1}$. Since $p|F_L$ we have $p|F_{p+1}$.
As $p|F_{p+1}$ and $p|F_{p-1}$ we have $p|(F_{p+1}-F_{p-1})=F_p$. Since $p$ divides
two consecutive Fibonacci numbers, $p$ divides all Fibonacci numbers - a contradiction.
So $p\equiv 1({\rm mod}\, L)$. The result follows from Lemma~\ref{Legendre}.

Similarly, assuming $p\equiv \pm 2({\rm mod}\, 5)$ gives $p\equiv -1 = (\frac{5}{p})({\rm mod}\, L)$, though we use Lemma~\ref{pdivideFibo2} instead of Lemma~\ref{pdivideFibo1}.
\end{proof}

In other words, consider the odd prime divisors of $F_L$  that are $\pm 1({\rm mod}\, L)$.
Those that are $1({\rm mod}\, L)$ are those that are  $\pm 1({\rm mod}\, 5)$. Those that are $-1({\rm mod}\, L)$
are those that are $\pm 2({\rm mod}\, 5)$. We are now ready to prove Theorem~\ref{MainTheorem}.

\begin{proof}
For each $i$ we have $p_i|F_L$ and $p_i\equiv \pm 1({\rm mod}\, L)$. So from Lemma~\ref{PrimeDivisorsPM1ModL}, for each $i$ we have $p_i\equiv (\frac{5}{p_i})
({\rm mod}\, L)$.
Taking the product of both sides over all $i$ gives $P\equiv (\frac{5}{P})
({\rm mod}\, L)$. Thus $L|(P-(\frac{5}{P}))$. From
Lemma~\ref{Fibodivide}
we have $F_L | F_{P-(5/P)}$. Since $P|F_L$ we get $P|F_{P-(\frac{5}{P})}$.
\end{proof}

Note, that the construction described in Theorem~\ref{MainTheorem} is related to, though not
the same as, the construction in \cite{CG} of Fibonacci pseudoprimes.

\section{The search for a Baillie-PSW pseudoprime}
\label{620}

There is a \$620 prize for a Baillie-PSW pseudoprime or a proof that none exists. That is an odd Fibonacci pseudoprime that is $\pm 2({\rm mod}\, 5)$
and also a base-2 pseudoprime. This problem was originally posed in \cite{PSW}.
Jan Feitsma and William Galway (see \cite{FG}) have computed all base 2 pseudoprimes up to $2^{64}$. Sam
Wagstaff has checked all of those to determine if there were any \$620 winners and there
were none (see  \cite{Po2}). The search is also described in \cite{BW,CG,CP,MK,Po1,PSW,SW}.

Fix a Fibonacci number $F_L$.
For $i=-1, 1$ we let $S_i$ be the set of odd prime divisors of $F_L$ that are congruent to
$i$ modulo $L$.
Recall that the primes in $S_{-1}$ are $\pm 2({\rm mod}\, 5)$ and
those in $S_1$ are $\pm 1({\rm mod}\, 5)$.
When possible, we created products of at least two distinct primes
from $S_{-1}\cup S_1$
such that the product contains an odd
number of primes from $S_{-1}$. That way the product is $\pm 2({\rm mod}\, 5)$.
For example, for $F_{258}$  we have $|S_{-1}|=5$ and $|S_1|=4$.
Thus we can create ${5\choose 1}\cdot (2^4-1)$ $+{5\choose 3}\cdot 2^4$ $+ {5\choose 5}\cdot
2^4$ $=251$ different products, all of which are odd Fibonacci pseudoprimes that are $\pm 2({\rm mod}\, 5)$.

For our search, we used the complete and partial factorizations of Fibonacci numbers
$F_L$ into prime divisors for $1\leq L\leq 9999$ found in \cite{FF}.
Using the construction described in the previous paragraph and these factorizations
we created approximately $2^{30}$ odd Fibonacci pseudoprimes that are $\pm 2({\rm mod}\, 5)$.
We then checked each to see if the Fibonacci pseudoprime $P$ is a base-2 pseudoprime, i.e.\ if $2^P\equiv 2({\rm mod}\, P)$.
Alas, none were base-2 pseudoprimes.

Some of the Fibonacci pseudoprimes we created are huge.
For example $F_{9967}$ is squarefree and all of its prime divisors are $\pm 1 ({\rm mod}\, 9967)$. Therefore $F_{9967}$
is a Fibonacci pseudoprime and it is
$-2({\rm mod}\, 5)$.
Note $F_{9967}\approx 2^{6918}$.

If $P=\prod_{i=1}^k p_i$ is a huge Fibonacci pseudoprime, then using the relatively fast repeated squares algorithm
to reduce $2^P({\rm mod}\, P)$ can still be quite slow.
Instead we note that by the Chinese Remainder Theorem,
$2^P\equiv 2({\rm mod}\, P)$ if and only if
$2^P\equiv 2({\rm mod}\, p_i)$ for each $i$.
We order the prime divisors of $P$ as $p_1 < p_2 < \ldots < p_k$.
First we reduce $2^P({\rm mod}\, p_1)$. If the remainder is not 2, as is usually the
case, then we know
$P$ is not a base-2 pseudoprime. If the remainder is 2, we do the same computation for
$p_2$ and so on. As soon as we get a remainder that is not 2, we can quit. If all $r$
remainders are, in fact 2, then $P$ is a base-2 pseudoprime (which
never occurred for us).

We can speed up the computation of $2^P({\rm mod}\, p_i)$ by noting that if
$P\equiv r_i ({\rm mod}\, p_i-1)$, with $0\leq r_i < p_i-1$, then
$2^P\equiv 2^{r_i}({\rm mod}\, p_i)$, which is a
consequence of Fermat's Little Theorem. So instead we can evaluate $2^{r_i}({\rm mod}\, p_i)$.
To determine $r_i$ we can iteratively multiply together $((p_1\cdot p_2)\cdot p_3)\ldots$,
reducing modulo $p_i-1$ after each multiplication. That way, the largest integer every appearing in the algorithm is at most $p_ip_k$.

\end{document}